\documentclass[12pt]{amsart}
\usepackage{amssymb}
\usepackage[all]{xy}

\newtheorem{theorem}{Theorem}[section]
\newtheorem{definition}[theorem]{Definition}
\newtheorem{proposition}[theorem]{Proposition}
\newtheorem{conjecture}[theorem]{Conjecture}

\begin{document}

\title[Isolated partial Hadamard matrices]{Isolated partial Hadamard matrices, and related topics}

\author[T. Banica, D. \"Ozteke, L. Pittau]{Teodor Banica, Duygu \"Ozteke and Lorenzo Pittau}
\address{Department of Mathematics, University of Cergy-Pontoise, F-95000 Cergy-Pontoise, France. {\tt teo.banica@gmail.com, Ozteked@live.fr, lorenzopittau@gmail.com}}

\subjclass[2010]{15B34 (46L65, 81R50)}
\keywords{Quantum permutation, Hadamard matrix, Matrix model}

\begin{abstract}
We study the isolated partial Hadamard matrices, under the assumption that the entries are roots of unity, or more generally, under the assumption that the combinatorics comes from vanishing sums of roots of unity. We first review the various conjectures on the subject, and then we present several new results, regarding notably the master Hadamard matrices, and the McNulty-Weigert construction. We discuss then the notion of isolation in some related contexts, of the magic unitary matrices, and of the quantum permutation groups, with a number of conjectures on the subject.
\end{abstract}

\maketitle

\section*{Introduction}

A complex Hadamard matrix is a square matrix $H\in M_N(\mathbb C)$ whose entries are on the unit circle, $H_{ij}\in\mathbb T$, and whose rows are pairwise orthogonal. The basic example of such a matrix is the Fourier one, $F_N=(w^{ij})$ with $w=e^{2\pi i/N}$:
$$F_N=\begin{pmatrix}
1&1&1&\ldots&1\\
1&w&w^2&\ldots&w^{N-1}\\
\ldots&\ldots&\ldots&\ldots&\ldots\\
1&w^{N-1}&w^{2(N-1)}&\ldots&w^{(N-1)^2}
\end{pmatrix}$$

The terminology here comes from the fact that $F_N$ is the matrix of the Fourier transform over the cyclic group $\mathbb Z_N$. More generally, associated to any finite abelian group is its Fourier matrix $F_G\in M_{|G|}(\mathbb C)$, which is a complex Hadamard matrix. For instance, the group $G=\mathbb Z_2^n$ produces the Walsh matrix $W_N\in M_N(\pm1)$, with $N=2^n$.

In general, the complex Hadamard matrices can be thought of as being ``generalized Fourier matrices''. They can be used in various situations, where a ``non-standard'' generalization of the discrete Fourier transform is needed. There are many applications of this philosophy, with the complex Hadamard matrices virtually appearing in any branch of quantum physics. We refer here to the paper of Tadej and \.Zyczkowski \cite{tz1}.

\bigskip

As an example here, for the reader who is familiar with operator algebras, according to the von Neumann philosophy \cite{mvo}, \cite{von}, one interesting question is that of studying the pairs of orthogonal MASA in type I factors, and such pairs are parametrized by the complex Hadamard matrices. See Haagerup \cite{ha1}, \cite{ha2}, Jones \cite{jon}, Popa \cite{pop}. 

This point of view can be further developed in the context of Woronowicz's compact quantum groups \cite{wo1}, \cite{wo2}, with the conclusion that the MASA combinatorics comes in fact from a quantum permutation group $G\subset S_N^+$. For a Fourier matrix $F_G$ this quantum group is $G$ itself, acting on itself, $G\subset S_{|G|}\subset S_{|G|}^+$, and in general, the correspondence $H\leftrightarrow G$ is something of ``Fourier type''. In addition, a direct link with statistical mechanics is expected to be possible via quantum groups, along the lines of \cite{jon}. See \cite{bbi}, \cite{bbs}.

\bigskip

An interesting question is that of understanding the structure of the complex Hadamard matrices, from a purely mathematical point of view. These 
 matrices from a real algebraic manifold, which appears as an intersection of smooth manifolds, as follows:
$$C_N=M_N(\mathbb T)\cap\sqrt{N}U_N$$

The local geometry of $C_N$ was studied by Tadej and \.Zyczkowski in \cite{tz2}, with an explicit result regarding the first order deformations of the Hadamard matrices $H\in C_N$. In the case of the Fourier matrix $H=F_N$, which is of particular interest, a finer result, regarding the plain deformations, was recently obtained by Nicoara and White in \cite{nwh}.

An alternative approach to these questions comes from the ``almost Hadamard matrix'' theory in \cite{bn2}, \cite{bnz}, which conjecturally allows one to recover $C_N$ via a gradient descent method, starting from the unitary group $U_N$. For details here, we refer to \cite{bn2}.

\bigskip

Summarizing, there is some ongoing work on the geometric structure of $C_N$, with some partial results being already available. From a physical perspective all this is certainly interesting, because one would expect the local geometry of $C_N$ around one of its points $H\in C_N$ to appear in the various physical contexts where $H$ appears.

As an example here, one interesting question is whether the invariants of the quantum permutation group $G\subset S_N^+$ associated to a given matrix $H\in C_N$ can be recaptured from the knowledge of the geometry of $C_N$ around $H$. This problem is not solved yet, and any solution here could probably have some applications to statistical mechanics.

\bigskip

Now by getting back to abstract mathematics, and to questions regarding the general structure of $C_N$, the ``simplest'' Hadamard matrices $H\in C_N$, at least from a geometric viewpoint, are those having the smallest possible neighborhood inside $C_N$. 

To be more precise, given $H\in C_N$, we can always multiply the rows and columns by arbitrary numbers in $\mathbb T$, and when these numbers are close to 1, we obtain in this way a family of trivial deformations of $H$. By slightly abusing the language, when the only deformations of $H$ are these trivial deformations, we will call $H$ ``isolated''.

When $H$ is dephased, in the sense that its first row and column contain only 1 entries, this isolation condition means precisely that $H$ is isolated in the usual sense, inside the dephased Hadamard matrix manifold $D_N\subset C_N$. In general, each matrix $H\in C_N$ gives rise to a dephased matrix $H'\in D_N$, and the isolation condition on $H$ corresponds to the fact that $H'\in D_N$ is isolated, in the usual sense. See \cite{tz1}, \cite{tz2}.

\bigskip

The isolation condition is in general difficult to check, but a useful criterion comes from \cite{tz2}.  Indeed, if $H\in C_N$ has only trivial order 1 deformations, which amounts in checking that a certain system of algebraic equations has only trivial solutions, then it must be isolated. The system of equations coming from \cite{tz2} is of course quite complicated, but in practice, it can be implemented on a computer, and succesfully used in this way.

We will advance here on these isolation questions, notably with some new results regarding the McNulty-Weigert construction \cite{mwe}, which covers many known examples of isolated complex Hadamard matrices. We will discuss as well another special class of matrices, namely the master Hadamard matrices \cite{aff}. Finally, we will discuss the notion of isolation in some related contexts, of the magic unitary matrices, and of the quantum permutation groups, with a number of conjectures on the subject.

\bigskip

Technically speaking, most of our results are under the assumption that the entries of the matrix are roots of unity, or more generally, that the combinatorics of the matrix comes from vanishing sums of roots of unity. Also, in order to obtain more evidence for our various conjectures, we will often work in the rectangular matrix setting, that of the partial Hadamard matrices, which will provide us with some useful flexibility. Of course, the partial Hadamard matrices are objects having their own interest, and there is a growing theory here \cite{ba2}, \cite{bsk}, \cite{hal}, \cite{dle}, to which the present paper adds.

\bigskip

As in the square matrix case, the partial Hadamard matrices are expected to be relevant in connection with a number of quantum physics questions. This subject is, however, largely unexplored. As an example of a potential application here, each partial Hadamard matrix $H\in M_{M\times N}(\mathbb C)$ produces via a matrix model a certain quantum partial permutation semigroup $G\subset\widetilde{S}_M^+$, and therefore provides us with useful information on the quantum semigroup $\widetilde{S}_M^+$ itself. This was explained some time ago in \cite{bsk}, and further work on the subject, in \cite{ba3} and in subsequent papers, has shown that the quantum spaces of type $\widetilde{S}_M^+$ are very basic examples of noncommutative homogeneous spaces, on which one can integrate via an analogue of the Weingarten formula. Thus, we have here a connection with noncommutative geometry in a Riemannian sense, and from this perspective, the partial Hadamard matrices correspond to the ``simplest matrix models for the simplest noncommutative homogeneous spaces, on which we can effectively integrate''. All this remains of course to be clarified, and we will partly advance here on these questions.

\bigskip

The paper is organized in three parts, as follows:

\bigskip

{\bf I.} In sections 1-2 we discuss a number of general conjectures regarding the isolated partial Hadamard matrices, in connection with the notion of defect, from \cite{tz2}, and with the notion of regularity, from \cite{bbs}. A part of the work here consists in formulating some extensions of the results in \cite{tz2}, in the rectangular matrix setting, and we will do this. Another part of the work consists in discussing the $N=7$ extension and the rectangular analogues of the $N=6$ classification results in \cite{bbs}. Here the subject is quite technical, but we will advance on these questions, our main statement being a precise conjecture.

\bigskip

{\bf II.} In sections 3-4 we discuss defect and isolation questions for two special types of matrices, namely the master Hadamard matrices, from \cite{aff}, and the McNulty-Weigert matrices, from \cite{mwe}. In both cases, our first contribution will be an extension of the general theory, from the square matrix setting to the rectangular matrix setting. We will also obtain, as main results here, an abstract formula for the defect of a master Hadamard matrix, and an explicit formula for the main series of McNulty-Weigert matrices. Further studying the McNulty-Weigert matrices remains of course the main question around.

\bigskip

{\bf III.} In sections 5-6 we study the quantum permutation groups $G\subset S_N^+$ associated to the complex Hadamard matrices $H\in M_N(\mathbb C)$, and more generally the quantum partial permutation semigroups $G\subset\widetilde{S}_M^+$ associated to the partial complex Hadamard matrices $H\in M_{M\times N}(\mathbb C)$, and we discuss the notion of isolation for them. Our main results here build on the previous work in \cite{bsk}, with a number of new contributions. Generally speaking, the notion of isolation in the quantum algebra setting appears to be a quite subtle one. We will try, however, at least to formulate some questions on this subject.

\section{Hadamard matrices}

We are interested in what follows in the complex Hadamard matrices $H\in M_N(\mathbb C)$, and notably in a number of conjectures regarding the ``isolated'' case. In order to comment on these conjectures, which, generally speaking, are quite difficult, it is convenient to allow a certain degree of flexibility in our formalism. To be more precise, we will use:

\begin{definition}
A partial complex Hadamard matrix is a matrix $H\in M_{M\times N}(\mathbb C)$ whose entries are on the unit circle, $H_{ij}\in\mathbb T$, and whose rows are pairwise orthogonal.
\end{definition}

As a basic example here, we have the various $M\times N$ submatrices of a given complex Hadamard matrix $H\in M_N(\mathbb C)$, with $M\leq N$. In general, these matrices are objects having their own interest, and in the real case $H\in M_{M\times N}(\pm1)$ for instance, there are several interesting results regarding them, including general structure results \cite{hal}, probabilistic counting results \cite{ba2}, \cite{dle} and relations with quantum algebra \cite{ba3}, \cite{bsk}.

For our purposes here, we are mostly interested in adopting the above framework in order to be able to provide more evidence for certain square matrix conjectures.

As in the square matrix case, we can use the following equivalence relation:

\begin{definition}
Two partial complex Hadamard matrices $H,K\in M_{M\times N}(\mathbb C)$ are called equivalent if one can pass from one to the other by permuting the rows and columns, or by multiplying the rows and columns by numbers in $\mathbb T$.
\end{definition}

Observe that any partial complex Hadamard matrix can be assumed, up to equivalence, to be ``dephased'', in the sense that its first row and column consist of $1$ entries only.

Consider now the real algebraic manifold $C_{M,N}$ formed by all the $M\times N$ partial complex Hadamard matrices. In order to discuss the geometry of this manifold, and notably the notion of isolation that we are interested in, we use the following conventions:
\begin{enumerate}
\item We denote by $M_x$ an unspecified neighborhood of a point in a manifold, $x\in M$.

\item We consider functions $f:M_x\to N_y$, which by definition satisfy $f(x)=y$.

\item For $q\in\mathbb T_1$, meaning $q\in\mathbb T$ close to $1$, we define $q^r$ with $r\in\mathbb R$ by $(e^{it})^r=e^{itr}$. 
\end{enumerate}

With these conventions made, we can talk about deformations, as follows:

\begin{definition}
Let $H\in C_{M,N}$ be a partial complex Hadamard matrix.
\begin{enumerate}
\item A deformation of $H$ is a smooth function $f:\mathbb T_1\to (C_{M,N})_H$.

\item The deformation is called ``affine'' if $f_{ij}(q)=H_{ij}q^{A_{ij}}$, with $A\in M_{M\times N}(\mathbb R)$.

\item We call ``trivial'' the deformations $f_{ij}(q)=H_{ij}q^{a_i+b_j}$, with $a\in\mathbb R^M,b\in\mathbb R^N$.
\end{enumerate}
\end{definition}

Here the adjective ``affine'' comes from $f_{ij}(e^{it})=H_{ij}e^{iA_{ij}t}$, because the function $t\to A_{ij}t$ which produces the exponent is indeed affine. As for the adjective ``trivial'', this comes from the fact that $f(q)=(H_{ij}q^{a_i+b_j})_{ij}$ is obtained from $H$ by multiplying the rows and columns by numbers in $\mathbb T$, so it is automatically partial Hadamard. See \cite{ba1}, \cite{tz1}.

We have $C_{M,N}=M_{M\times N}(\mathbb T)\cap\sqrt{N}U_{M,N}$, where $U_{M,N}\subset M_{M\times N}(\mathbb C)$ is the set of matrices having all rows of norm 1, and pairwise orthogonal. This remark leads us to:

\begin{definition}
Associated to a point $H\in C_{M,N}$ are:
\begin{enumerate}
\item The enveloping tangent space: $\widetilde{T}_HC_{M,N}=T_HM_{M\times N}(\mathbb T)\cap T_H\sqrt{N}U_{M,N}$.

\item The tangent cone $T_HC_{M,N}$: the set of tangent vectors to the deformations of $H$.

\item The affine tangent cone $T_H^\circ C_{M,N}$: same as above, using affine deformations only.

\item The trivial tangent cone $T_H^\times C_{M,N}$: as above, using trivial deformations only.
\end{enumerate}
\end{definition}

Observe that $\widetilde{T}_HC_{M,N},T_HC_{M,N}$ are real vector spaces, and that $T_HC_{M,N},T_H^\circ C_{M,N}$ are two-sided cones ($\lambda\in\mathbb R,A\in T\implies\lambda A\in T$). Also, we have inclusions as follows:
$$T_H^\times C_{M,N}\subset T_H^\circ C_{M,N}\subset T_HC_{M,N}\subset\widetilde{T}_HC_{M,N}$$

Since $\widetilde{T}_HC_{M,N}$ is a real vector space, of particular interest is the computation of its dimension $d(H)=\dim(\widetilde{T}_HC_{M,N})$, called defect of $H$. See \cite{ba1}, \cite{nwh}, \cite{tz2}. We have:

\begin{proposition}
Let $H\in C_{M,N}$, and pick $K\in\sqrt{N}U_N$ extending $H$.
\begin{enumerate}
\item $\widetilde{T}_HC_{M,N}\simeq\{A\in M_{M\times N}(\mathbb R)|\sum_kH_{ik}\bar{H}_{jk}(A_{ik}-A_{jk})=0,\forall i,j\}$.

\item $\widetilde{T}_HC_{M,N}\simeq\{E=(X\ Y)\in M_{M\times N}(\mathbb C)|X=X^*,(EK)_{ij}\bar{H}_{ij}\in\mathbb R,\forall i,j\}$.
\end{enumerate}
The correspondence $A\to E$ is given by $E_{ij}=\sum_kH_{ik}\bar{K}_{jk}A_{ik}$, $A_{ij}=(EK)_{ij}\bar{H}_{ij}$.
\end{proposition}

\begin{proof}
These results are well-known in the square case, and the proofs extend:

(1) In the square case, $M=N$, this formula is due to Tadej and \.Zyczkowski \cite{tz1}, and was reproved in \cite{ba1}, by using a slightly different approach. The proof in \cite{ba1} extends in a straightforward way to the general rectangular case, $M\leq N$.

(2) Set $R_{ij}=A_{ij}H_{ij}$ and $E=RK^*$. The correspondence $A\to R\to E$ is then bijective, and we have $E_{ij}=\sum_kH_{ik}\bar{K}_{jk}A_{ik}$. The system of equations in (1) becomes $E_{ij}=\bar{E}_{ji}$ for any $i,j$ with $j\leq M$, which shows that we must have $E=(X\ Y)$ with $X=X^*$, and the condition $A_{ij}\in\mathbb R$ corresponds to the condition $(EK)_{ij}\bar{H}_{ij}\in\mathbb R$, as claimed.
\end{proof}

As an illustration, in the real case we obtain the following result:

\begin{proposition}
For any partial Hadamard matrix $H\in M_{M\times N}(\pm1)$ we have:
$$\widetilde{T}_HC_{M,N}\simeq M_M(\mathbb R)^{symm}\oplus M_{M\times(N-M)}(\mathbb R)$$
In particular, for any $N\times N$ Hadamard matrix we have $\widetilde{T}_HC_N\simeq M_N(\mathbb R)^{symm}$.
\end{proposition}

\begin{proof}
We use Proposition 1.5 (2). Since $H$ is now real we can pick $K\in\sqrt{N}U_N$ extending it to be real too, and with nonzero entries, so the last condition appearing there, namely $(EK)_{ij}\bar{H}_{ij}\in\mathbb R$, simply tells us that $E$ must be real. Thus we have:
$$\widetilde{T}_HC_{M,N}\simeq\{E=(X\ Y)\in M_{M\times N}(\mathbb R)|X=X^*\}$$

But this is the formula in the statement, and we are done.
\end{proof}

Let us get now to the subject that we are interested in, namely isolation. A matrix $H\in C_{M,N}$ cannot be isolated, simply because the space of its Hadamard equivalents provides a copy $\mathbb T^{MN}\subset C_{M,N}$, passing through $H$. However, if we restrict the attention to the matrices which are dephased, the notion of isolation makes sense, and we have:

\begin{proposition}
Let $d(H)=\dim(\widetilde{T}_HC_{M,N})$.
\begin{enumerate}
\item This number, called undephased defect of $H$, satisfies $d(H)\geq M+N-1$.

\item If $d(H)=M+N-1$ then $H$ is isolated inside the dephased quotient $C_{M,N}\to D_{M,N}$.
\end{enumerate}
\end{proposition}

\begin{proof}
Once again, these are well-known facts from the square case, which extend:

(1) We have indeed $\dim(T_H^\times C_{M,N})=M+N-1$, and since the tangent vectors to these trivial deformations belong to $\widetilde{T}_HC_{M,N}$, this gives the result.

(2) Since $d(H)=M+N-1$, the inclusions $T_H^\times C_{M,N}\subset T_HC_{M,N}\subset\widetilde{T}_HC_{M,N}$ must be equalities, and from $T_HC_{M,N}=T_H^\times C_{M,N}$ we obtain the result.
\end{proof}

Finally, still at the theoretical level, we have the following conjecture:

\begin{conjecture}[Isolation]
An isolated matrix $H\in D_{M,N}$ must have minimal defect, namely $d(H)=M+N-1$.
\end{conjecture}

In other words, the conjecture is that if $H\in C_{M,N}$ has only trivial first order deformations, then it has only trivial deformations at any order, including  at $\infty$. 

In the square matrix case this statement comes with solid evidence, all the known examples of complex Hadamard matrices $H\in C_N$ having non-minimal defect being known to admit one-parameter deformations. For more on this subject, see \cite{tz1}, \cite{tz2}.

The above statement is actually expected to be part of a more general phenomenon, in view of the recent results in \cite{nwh}, which show that for the Fourier matrix $F_N$, the first order deformations all come from one-parameter deformations. Thus, there should be a stronger conjecture here, covering as well certain non-isolated matrices.

Let us discuss now some examples of isolated Hadamard matrices, and provide some evidence for Conjecture 1.8. We are interested in the following matrices:

\begin{definition}
The truncated Fourier matrix $F_{S,G}$, with $G$ being a finite abelian group, and with $S\subset G$ being a subset, is constructed as follows:
\begin{enumerate}
\item Given $N\in\mathbb N$, we set $F_N=(w^{ij})_{ij}$, where $w=e^{2\pi i/N}$.

\item Assuming $G=\mathbb Z_{N_1}\times\ldots\times\mathbb Z_{N_s}$, we set $F_G=F_{N_1}\otimes\ldots\otimes F_{N_s}$.

\item We let $F_{S,G}$ be the submatrix of $F_G$ having $S\subset G$ as row index set. 
\end{enumerate}
\end{definition}

Observe that $F_N$ is the Fourier matrix of the cyclic group $\mathbb Z_N$. More generally, $F_G$ is the Fourier matrix of the finite abelian group $G$. Observe also that $F_{G,G}=F_G$.

We can compute the defect of $F_{S,G}$ by using Proposition 1.5, and we obtain:

\begin{proposition}
For a truncated Fourier matrix $F=F_{S,G}$ we have the formula
$$\widetilde{T}_FC_{M,N}=\left\{A\in M_{M\times N}(\mathbb R)\Big|P=AF^t\ {\rm satisfies}\ P_{ij}=P_{i+j,j}=\bar{P}_{i,-j},\forall i,j\right\}$$
where $M=|S|,N=|G|$, and with all the indices being regarded as group elements. 
\end{proposition}

\begin{proof}
We use Proposition 1.5 (1). The equations there are $\sum_kF_{ik}\bar{F}_{jk}(A_{ik}-A_{jk})=0$, and since for $F=F_{S,G}$ we have $F_{ik}\bar{F}_{jk}=(F^t)_{k,i-j}$, we obtain:
$$\widetilde{T}_FC_{M,N}=\left\{A\in M_{M\times N}(\mathbb R)\big|(AF^t)_{i,i-j}=(AF^t)_{j,i-j},\forall i,j\right\}$$

Now observe that for an arbitrary matrix $P\in M_M(\mathbb C)$, we have:
$$P_{i,i-j}=P_{j,i-j},\forall i,j\iff P_{i+j,i}=P_{ji},\forall i,j\iff P_{i+j,j}=P_{ij},\forall i,j$$

We therefore conclude that we have the following equality:
$$\widetilde{T}_FC_{M,N}=\left\{A\in M_{M\times N}(\mathbb R)\big|
P=AF^t\ {\rm satisfies}\ P_{ij}=P_{i+j,j},\forall i,j\right\}$$

Now observe that with $A\in M_{M\times N}(\mathbb R)$ and $P=AF^t\in M_M(\mathbb C)$ as above, we have:
$$\bar{P}_{ij}=\sum_kA_{ik}(F^*)_{kj}=\sum_kA_{ik}(F^t)_{k,-j}=P_{i,-j}$$

Thus, we obtain the formula in the statement, and we are done. 
\end{proof}

As an illustration, in the square matrix case we obtain the following formulae, known since the work of Tadej and \.Zyczkowski \cite{tz2}:

\begin{proposition}
For a Fourier matrix $F=F_G$, we have:
\begin{enumerate}
\item $\widetilde{T}_FC_N=\{PF^*|P_{ij}=P_{i+j,j}=\bar{P}_{i,-j}\}$.

\item $d(F)=\sum_{g\in G}[G:<g>]$.

\item $d(F)$ is also the number of $1$ entries of $F$.

\item For $G=\mathbb Z_N$ with $N=\prod_ip_i^{a_i}$ we have $d(F)=N\prod_i(1+a_i-\frac{a_i}{p_i})$.
\end{enumerate}
\end{proposition}

\begin{proof}
We use the formula in Proposition 1.10, with $S=G$:

(1) This is clear, because the correspondence $A\to P=AF^t$ is now bijective.

(2) The proof here uses an identification of real vector spaces, as follows:
$$\{P\in M_N(\mathbb C)|P_{ij}=P_{i+j,j}=\bar{P}_{i,-j}\}\simeq\bigoplus_{g\in G}C(G/<g>,\mathbb R)$$

Indeed, if we let $G_2=\{g\in G|2g=0\}$, and then choose a partition of type $G=G_2\sqcup X\sqcup(-X)$, the formula giving the above identification is $P=\oplus P_g$, with:
$$P_g(i)=
\begin{cases}
P_{ij}&(j\in G_2)\\
Re(P_{ij})&(j\in X)\\
Im(P_{ij})&(j\in -X)
\end{cases}$$

With this identification in hand, the result follows from (1).

(3) This observation follows from (2), and from the definition of $F$.

(4) This formula, due to Tadej and \.Zyczkowski, follows from (2) or (3).
\end{proof}

Now back to our isolation questions, as a first, well-known result, we have:

\begin{proposition}
The Fourier matrix $F_G$ of a finite abelian group $G$ is isolated precisely when we have $G=\mathbb Z_p$, with $p$ prime.
\end{proposition}

\begin{proof}
This follows indeed from the formulae found in Proposition 1.11 above.
\end{proof}

Let us restrict now the attention to the purely rectangular case, $M<N$, and try to find some explicit examples here, of truncated Fourier type. For this purpose, we can use the following slightly improved version of Proposition 1.10 above:

\begin{theorem}
The defect of $F=F_{S,G}$ is the number $\dim(K)+\dim(I)$, where
\begin{eqnarray*}
K&=&\left\{A\in M_{M\times N}(\mathbb R)\Big|AF^t=0\right\}\\
I&=&\left\{P\in L_M\Big|\exists A\in M_{M\times N}(\mathbb R),P=AF^t\right\}
\end{eqnarray*}
where $L_M=\left\{P\in M_M(\mathbb C)\big|P_{ij}=P_{i+j,j}=\bar{P}_{i,-j},\forall i,j\right\}$. 
\end{theorem}

\begin{proof}
We use Proposition 1.10. With the notations there, and with $L_M$ being as above, we have a linear map $\Phi:\widetilde{T}_FC_{M,N}\to L_M$ given by $\Phi(A)=AF^t$, which gives:
$$\dim(\widetilde{T}_FC_{M,N})=\dim(\ker\Phi)+\dim({\rm Im}\,\Phi)$$

Now since the spaces on the right are precisely those in the statement, $\ker\Phi=K$ and ${\rm Im}\, \Phi=I$, by applying Proposition 1.10 we obtain the result.
\end{proof}

In order to look now for isolated matrices, the first remark is that since a deformation of $F_G$ will produce a deformation of $F_{S,G}$ too, we must  restrict the attention to the case where $G=\mathbb Z_p$, with $p$ prime. And here, we have the following conjecture:

\begin{conjecture}[Truncation]
There exists a constant $\varepsilon>0$ such that $F_{S,p}$ is isolated, for any $p$ prime, once $S\subset\mathbb Z_p$ satisfies $|S|\geq(1-\varepsilon)p$.
\end{conjecture}

In principle this conjecture can be approached by using the formula in Theorem 1.13 above, and we have for instance evidence towards the fact that $F_{p-1,p}$ should be always isolated, that $F_{p-2,p}$ should be isolated too, provided that $p$ is big enough, and so on. However, finding a number $\varepsilon>0$ as above looks like a quite difficult question.

\section{Roots of unity, regularity}

We restrict now the attention to the case where the entries are roots of unity, or more generally to the case where the ``row combinatorics'' comes from the roots of unity.

In order to introduce these classes of matrices, we use:

\begin{definition}
A cycle is a full sum of roots of unity, possibly rotated by a scalar, and taken in a formal sense. A sum of cycles is a formal sum of cycles.
\end{definition}

The actual sum of a cycle, or of a sum of cycles, is of course 0. This is why the word ``formal'' is there, for reminding us that we are working with formal sums.

As an example, here is a sum of cycles, with $w=e^{2\pi i/6}$, and with $|q|=1$:
$$1+w^2+w^4+qw+qw^4=0$$

As a non-example now, here is a vanishing sum of roots of unity which is not a sum of cycles (at least with positive coefficients, as defined above), with $w=e^{2\pi i/30}$:
$$w^5+w^6+w^{12}+w^{18}+w^{24}+w^{25}=0$$

The following key result about cycles is due to Lam and Leung \cite{lle}:

\begin{proposition}
Let $s=p_1^{a_1}\ldots p_k^{a_k}$, and assume that $\lambda_i\in\mathbb Z_l$ satisfy $\lambda_1+\ldots+\lambda_N=0$.
\begin{enumerate}
\item $\sum\lambda_i$ is a sum of cycles, with $\mathbb Z$ coefficients.

\item If $k\leq 2$ then $\sum\lambda_i$ is a sum of cycles (with $\mathbb N$ coefficients).

\item At $k\geq 3$ then $\sum\lambda_i$ might not decompose as a sum of cycles.

\item $\sum\lambda_i$ has the same length as a sum of cycles: $N\in p_1\mathbb N+\ldots+p_k\mathbb N$.
\end{enumerate}
\end{proposition}

\begin{proof}
The first three assertions are elementary, and the last one is from \cite{lle}:

(1) This is clear at $k=1$, and is quite elementary as well at $k=2$. See \cite{lle}.

(2) This is a well-known result, which follows from basic number theory.

(3) The simplest counterexample here is the one given above, with $w=e^{2\pi i/30}$.

(4) This is actually quite a deep result, due to Lam and Leung \cite{lle}.
\end{proof}

We can introduce the special classes of partial complex Hadamard matrices that we will be mostly interested in, as follows:

\begin{definition}
A  partial complex Hadamard matrix $H\in M_{M\times N}(\mathbb T)$ is called:
\begin{enumerate}
\item Of Butson type, if all the entries are roots of unity of finite order. 

\item Regular, if the scalar products between rows decompose as sums of cycles.
\end{enumerate}
\end{definition}

As usual, these notions are well-known in the square matrix case, with (1) being due to Butson \cite{but}, and with (2) being introduced in \cite{bbs}. We should mention that there is some notational clash here, with the notion of regularity, which is used in the purely algebraic literature in order to designate the bistochastic matrices. In this paper we use the above notion of regularity, and we call bistochastic the bistochastic matrices.

We have the following conjectural statement, from \cite{bbs}:

\begin{conjecture}[Regularity]
Any Butson matrix $H\in M_N(\mathbb C)$ is regular.
\end{conjecture}

In other words, the conjecture is that a ``tricky vanishing sum'' of roots of unity, like the $l=30$ one given after Definition 2.1 above, cannot be used in order to construct a complex Hadamard matrix. This is a quite difficult question, coming with substantial computer evidence. We have so far no idea on how to approach it. See \cite{bbs}.

In relation now with the geometric considerations from section 1 above, we have:

\begin{conjecture}[Deformation]
Any regular matrix $H\in M_N(\mathbb C)$ is an affine deformation of a Butson matrix.
\end{conjecture}

The problem here is on how to define a Butson matrix $H^1\in M_N(\mathbb C)$ which can produce $H$ itself via an affine deformation. The first thought would go to the idea ``dephase $H$, then look at the scalar products $<H_1,H_i>$, and set all the $\lambda$-parameters which multiply the various cycles to be 1''. But this is a quite complicated task, because the sum-of-cycles decomposition is not unique, and because the $\lambda$-parameters are not unique either.

In the rectangular case now, Conjecture 2.4 and Conjecture 2.5 both cannot hold. However, we have here the following related statement:

\begin{conjecture}[Weak isolation]
A dephased partial Hadamard matrix $H\in M_{M\times N}(\mathbb C)$ which is regular and isolated must be of Butson type.
\end{conjecture}

Observe that in the square case, this statement is weaker than Conjecture 2.5.

Observe also that, since by Proposition 1.7 (2) minimal defect implies isolation, a slightly weaker conjecture would be that a regular Hadamard matrix $H\in M_{M\times N}(\mathbb C)$ having minimal defect, $d(H)=M+N-1$, must be of Butson type. However, in view of Conjecture 1.8, this weaker statement is probably equivalent to the original one.

As a first piece of evidence for this conjecture, we have:

\begin{proposition}
The weak isolation conjecture holds at $M=2$.
\end{proposition}

\begin{proof}
Consider indeed a dephased regular matrix $H\in M_{2\times N}(\mathbb C)$. This matrix must look at follows, with $q_1,\ldots,q_s\in\mathbb T$, and with $C_1,\ldots,C_s$ being full sums of roots of unity:
$$H=\begin{pmatrix}
1&\ldots&1\\
q_1C_1&\ldots&q_sC_s
\end{pmatrix}$$

Now by assuming that our matrix is isolated, we can have only one cycle here, $s=1$. But since $H$ was assumed to be dephased, we must have $q_1=1$, and we are done.
\end{proof}

At $M=3$ or higher, providing evidence for the weak isolation conjecture, in the purely rectangular framework, is a quite non-trivial task. We will be back later on to this.

Let us discuss now the weak isolation conjecture in the square matrix case, with a complete discussion at $N\leq6$, and with a study, which is partly conjectural, at $N=7$.

Let us first recall the following key result, due to Haagerup \cite{ha1}:

\begin{proposition}
The only complex Hadamard matrices at $N=2,3,4,5$ are 
$$F_2,F_3,F_{2,2}^q,F_5$$
where $F_{2,2}^q$ is a certain one-parameter deformation of $F_{2,2}=F_2\otimes F_2$. 
\end{proposition}

\begin{proof}
The results at $N=2,3,4$ are elementary, with $F_{2,2}^q$ appearing via the Di\c t\u a deformation method \cite{dit}. As for the $N=5$ result, this is something non-trivial. See \cite{ha1}.
\end{proof}

At $N=6$ now, the complex Hadamard matrices are not fully classified. However, in the regular case we have the following result, from \cite{bbs}:

\begin{proposition}
The only regular Hadamard matrices at $N=6$ are the deformations $F_{2,3}^{q,r},F_{3,2}^{q,r}$ of the Fourier matrix $F_6$, cf. \cite{dit}, and the matrices $H_6^q,T_6$ from \cite{ha1}, \cite{tao}.
\end{proposition}

\begin{proof}
This result is from \cite{bbs}, the idea being as follows. First, the possible cycle lengths can be either $3+3$, or $2+2+2$. In the case where all scalar products are of type $3+3$, a direct computation leads to $T_6$. Also, in case where all scalar products are of type $2+2+2$, a direct computation leads to $H_6^q$. Finally, in the ``mixed'' case, the idea is that the ``cycle graph'' of our matrix, having as vertices the 6 rows, and with the edges colored black or white, according to the two possible cases  $3+3$ and $2+2+2$, can take only two values, and these two values lead to the matrices $F_{2,3}^{q,r},F_{3,2}^{q,r}$. See \cite{bbs}.
\end{proof}

In connection now with our questions, observe that the above results show that the weak isolation conjecture holds in the square matrix case, at any $N\leq6$. In order to discuss now the $N=7$ case, we must introduce the Petrescu matrix \cite{pet}.

This matrix, originally found in \cite{pet}, and which was later on put in a more conceptual framework by Sz\"oll\H{o}si in \cite{szo}, is as follows, with $w=e^{2\pi i/6}$, and with $q\in\mathbb T$: 
$$P_7^q
=\begin{pmatrix}
-q&q&w&1&w&1&w\\
q&-q&w&1&1&w&w\\
w&w&-w&1&w&w&1\\
1&1&1&-1&w&w&w\\
w&1&w&w&-\bar{q}w&\bar{q}w&1\\
1&w&w&w&\bar{q}w&-\bar{q}w&1\\
w&w&1&w&1&1&-1
\end{pmatrix}$$

Observe that the cycle structure of this matrix is $3+2+2$. The point now is that, according to the various computations that we have, this matrix seems to be the only one having cycle structure $3+2+2$. Moreover, together with some standard number theoretic work on the possible cycle structures of the regular $7\times7$ matrices, and with some results from \cite{hsc} as well, regarding the matrices $F_p$ with $p$ prime, we are led to:

\begin{conjecture}
$F_7,P_7^q$ are the only regular complex Hadamard matrices at $N=7$.
\end{conjecture}

As a conclusion to all these considerations, let us formulate:

\begin{theorem}
The weak isolation conjecture holds for the square matrices:
\begin{enumerate}
\item At $N=2,3,4,5,6$.

\item At $N=7$ as well, modulo Conjecture 2.10.
\end{enumerate}
\end{theorem}

\begin{proof}
Here (1) follows from Proposition 2.8 and Proposition 2.9, and (2) follows from the above-mentioned considerations.
\end{proof}

\section{Special matrices}

In this section and in the next one we discuss some general constructions of isolated matrices, and a number of related structure and classification questions.

We begin our study with some generalities on the tensor products, and on the related notion of Di\c t\u a deformation \cite{dit}. We first have the following result, which is standard:

\begin{proposition}
Given two complex Hadamard matrices $K,L$, we have:
$$d(K\otimes L)\geq d(K)d(L)$$
In particular, assuming $K,L\neq(1)$, the matrix $K\otimes L$ cannot be isolated.
\end{proposition}

\begin{proof}
The inequality in the statement is well-known, and follows for instance from the fact that, with terminology from section 1, the enveloping tangent space at $K\otimes L$ contains the tensor product of the enveloping tangent spaces at $K,L$. See \cite{ba1}, \cite{tad}.
\end{proof}

In view of the above result, a first natural question is that of finding the pairs of matrices $(K,L)$ having the property that $d(K\otimes L)\geq d(K)d(L)$ is sharp. We have here:

\begin{proposition}
In the Fourier matrix case $K=F_G,L=F_H$ we have 
$$d(K\otimes L)=d(K)d(L)$$
precisely when the orders of the finite abelian groups $G,H$ satisfy $(|G|,|H|)=1$.
\end{proposition}

\begin{proof}
Assuming $K=F_G,L=F_H$, we have $K\otimes L=F_{G\times H}$, and the result, which is well-known, follows by examining the number of 1 entries of the matrix $F_{G\times H}$.
\end{proof}

Yet another natural question emerging from Proposition 3.1 is that of finding the tensor products $K\otimes L$ which have minimal possible defect. We have here:

\begin{proposition}
The tensor products of Fourier matrices $F_G\otimes F_H$ having minimal defect are precisely those of the form $F_p\otimes F_q$, with $p\neq q$ both prime.
\end{proposition}

\begin{proof}
According to Proposition 3.1 above, we have the following inequality:
$$d(F_G\otimes F_H)\geq d(F_G)d(F_H)\geq(2|G|-1)(2|H|-1)$$

In order for the first inequality to hold, we know from Proposition 3.2 that we must have $(|G|,|H|)=1$. Also, in order for the second inequality to hold, we know from Proposition 1.12 that we must have $G=\mathbb Z_p$ and $H=\mathbb Z_q$, with both $p,q$ being prime. By combining these two requirements, we obtain the condition in the statement.
\end{proof}

Let us discuss now isolation questions for an interesting class of partial Hadamard matrices, namely the ``master'' ones, introduced in \cite{aff}, as follows:

\begin{definition}
A master Hadamard matrix is a partial Hadamard matrix of the form $H_{ij}=\lambda_i^{n_j}$, with $\lambda_i\in\mathbb T,n_j\in\mathbb R$. The associated ``master function'' is $f(z)=\sum_jz^{n_j}$.
\end{definition}

Observe that with $\lambda_i=e^{im_i}$ we have $H_{ij}=e^{im_in_j}$. In the square case, the basic example of such a matrix is the Fourier matrix $F_N$, having master function $f(z)=\frac{z^N-1}{z-1}$.

Observe that, in terms of $f$, the Hadamard condition on $H$ is simply:
$$f\left(\frac{\lambda_i}{\lambda_j}\right)=N\delta_{ij}$$

These matrices were introduced in \cite{aff}, the motivating remark there being the fact that $R=\sum_{ij}e_{ij}\otimes\Lambda^{n_i-n_j}$ defines a representation of the Temperley-Lieb algebra.

At the level of examples now, the first observation, from \cite{aff}, is that the standard $4\times 4$ complex Hadamard matrices are, with 2 exceptions, master Hadamard matrices:

\begin{proposition}
The following complex Hadamard matrix, with $|q|=1$,
$$F_{2,2}^q=\begin{pmatrix}
1&1&1&1\\
1&-1&1&-1\\
1&q&-1&-q\\
1&-q&-1&q
\end{pmatrix}$$
is a master Hadamard matrix, for any $q\neq\pm1$.
\end{proposition}

\begin{proof}
We use the exponentiation convention $(e^{it})^r=e^{itr}$ for $t\in[0,2\pi)$ and $r\in\mathbb R$. Since $q^2\neq1$, we can find $k\in\mathbb R$ such that $q^{2k}=-1$, and so our matrix becomes:
$$F_{2,2}^q=\begin{pmatrix}
1^0&1^1&1^{2k}&1^{2k+1}\\
(-1)^0&(-1)^1&(-1)^{2k}&(-1)^{2k+1}\\
q^0&q^1&q^{2k}&q^{2k+1}\\
(-q)^0&(-q)^1&(-q)^{2k}&(-q)^{2k+1}\\
\end{pmatrix}$$

Now if we pick $\lambda\neq1$ and write $1=\lambda^x,-1=\lambda^y,q=\lambda^z,-q=\lambda^t$, we are done.
\end{proof}

We recall that, given two Hadamard matrices $H\in M_N(\mathbb C)$ and $K\in M_M(\mathbb C)$, the Di\c t\u a deformation of $H\otimes K$ with parameter matrix $Q\in M_{N\times M}(\mathbb T)$ is given by:
$$H\otimes_QK=(Q_{ib}H_{ij}K_{ab})_{ia,jb}$$

It is known, and easy to check, that this matrix is indeed complex Hadamard. As a basic example, we have $F_{2,2}^q=F_2\otimes_QF_2$, where $Q=(^1_1{\ }^1_q)$. See \cite{dit}. 

We have the following generalization of Proposition 3.5, once again from \cite{aff}:

\begin{proposition}
$F_N\otimes_QF_M$ is a master Hadamard matrix, whenever $Q\in M_{N\times M}(\mathbb T)$ is given by $Q_{ib}=q^{i(Mp_b+b)}$, where $q=e^{2\pi i/MNk}$ with $k\in\mathbb N$, and $p_0,\ldots,p_{M-1}\in\mathbb R$.
\end{proposition}

\begin{proof}
The main construction in \cite{aff} is, in terms of master functions, as follows:
$$f(z)=f_N(z^{Mk})f_M(z)$$

Here $k\in\mathbb N$, and the functions on the right are by definition as follows:
$$f_N(z)=\sum_iz^{Nr_i+i}\quad\quad f_M(z)=\sum_az^{Mp_a+a}$$

The proof uses the eigenvalues $\lambda_{ia}=q^iw^a$, where $w=e^{2\pi i/M}$, and where $q^{Mk}=\nu$, where $\nu^N=1$. Observe now that, according to $f(z)=f_N(z^{Mk})f_M(z)$, the exponents are:
$$n_{jb}=Mk(Nr_j+j)+Mp_b+b$$

Thus the associated master Hadamard matrix is given by:
$$H_{ia,jb}=(q^iw^a)^{Mk(Nr_j+j)+Mp_b+b}=\nu^{ij}q^{i(Mp_b+b)}w^{a(Mp_b+b)}=\nu^{ij}w^{ab}q^{i(Mp_b+b)}$$

Now since $(F_N\otimes F_M)_{ia,jb}=\nu^{ij}w^{ab}$, we get $H=F_N\otimes_QF_M$ with $Q_{ib}=q^{i(Mp_b+b)}$, as claimed. Observe that $Q$ itself is a ``master matrix'', because the indices split.
\end{proof}

In view of the above examples, and of the lack of other known examples of master Hadamard matrices, in the square case, the following conjecture was made in \cite{aff}:

\begin{conjecture}[Master Hadamard conjecture]
In the square matrix case, the master Hadamard matrices appear as Di\c t\u a deformations of $F_N$.
\end{conjecture}

There is a relation here with the notions of defect and isolation, that we would like to discuss now. First, we have the following defect computation:

\begin{theorem}
The defect of a master Hadamard matrix is given by
$$d(H)=\dim_\mathbb R\left\{B\in M_N(\mathbb C)\Big|\bar{B}=\frac{1}{N}BL, (BR)_{i,ij}=(BR)_{j,ij}\ \forall i,j\right\}$$
where $L_{ij}=f(\frac{1}{\lambda_i\lambda_j})$ and $R_{i,jk}=f(\frac{\lambda_j}{\lambda_i\lambda_k})$, $f$ being the master function.
\end{theorem}

\begin{proof}
The first order deformation equations are $\sum_kH_{ik}\bar{H}_{jk}(A_{ik}-A_{jk})=0$. With $H_{ij}=\lambda_i^{n_j}$ we have $H_{ij}\bar{H}_{jk}=(\lambda_i/\lambda_j)^{n_k}$, and so the defect is given by:
$$d(H)=\dim_\mathbb R\left\{A\in M_N(\mathbb R)\Big|\sum_kA_{ik}\left(\frac{\lambda_i}{\lambda_j}\right)^{n_k}=\sum_kA_{jk}\left(\frac{\lambda_i}{\lambda_j}\right)^{n_k}\ \forall i,j\right\}$$

Now, pick $A\in M_N(\mathbb C)$ and set $B=AH^t$, so that $A=\frac{1}{N}B\bar{H}$. First, we have:
$$A\in M_N(\mathbb R)\iff B\bar{H}=\bar{B}H\iff \bar{B}=\frac{1}{N}B\bar{H}H^*$$

On the other hand, the matrix on the right is given by:
$$(\bar{H}H^*)_{ij}=\sum_k\bar{H}_{ik}\bar{H}_{jk}=\sum_k(\lambda_i\lambda_j)^{-n_k}=L_{ij}$$

Thus $A\in M_N(\mathbb R)$ if and only the condition $\bar{B}=\frac{1}{N}BL$ in the statement is satisfied. Regarding now the second condition on $A$, observe that with $A=\frac{1}{N}B\bar{H}$ we have:
$$\sum_kA_{ik}\left(\frac{\lambda_i}{\lambda_j}\right)^{n_k}=\frac{1}{N}\sum_{ks}B_{is}\left(\frac{\lambda_i}{\lambda_j\lambda_s}\right)^{n_k}=\frac{1}{N}\sum_sB_{is}R_{s,ij}=\frac{1}{N}(BR)_{i,ij}$$

Thus the second condition on $A$ reads $(BR)_{i,ij}=(BR)_{j,ij}$, and we are done.
\end{proof}

We can now formulate our isolation conjecture, in the square matrix case:

\begin{conjecture}[Isolated master Hadamard conjecture]
The only isolated square master Hadamard matrices are the Fourier matrices $F_p$, with $p$ prime.
\end{conjecture}

In the rectangular matrix case the situation is more complicated, due to several ``exotic'' examples, and we were unable so far to formulate a conjecture here.

\section{McNulty-Weigert matrices}

We discuss in this section a key construction of isolated Hadamard matrices, due to McNulty and Weigert \cite{mwe}. The starting observation is as follows:

\begin{proposition}
Assuming that $K\in M_{M\times N}(\mathbb C)$ is Hadamard, so is the matrix
$$H_{ia,jb}=\frac{1}{\sqrt{Q}}K_{ij}(L_i^*R_j)_{ab}$$
provided that $\{L_1,\ldots,L_M\}\subset\sqrt{Q}U_Q$ and $\{R_1,\ldots,R_N\}\subset\sqrt{Q}U_Q$ are such that each of the matrices $\frac{1}{\sqrt{Q}}L_i^*R_j\in\sqrt{Q}U_Q$, with $i=1,\ldots,M$ and $j=1,\ldots,N$, is Hadamard.
\end{proposition}

\begin{proof}
The check of the unitarity is done as in the square case \cite{mwe}, as follows:
\begin{eqnarray*}
<H_{ia},H_{kc}>
&=&\frac{1}{Q}\sum_{jb}K_{ij}(L_i^*R_j)_{ab}\bar{K}_{kj}\overline{(L_k^*R_j)}_{cb}\\
&=&\sum_jK_{ij}\bar{K}_{kj}(L_i^*L_k)_{ac}\\
&=&N\delta_{ik}(L_i^*L_k)_{ac}=NQ\delta_{ik}\delta_{ac}
\end{eqnarray*}

The entries being in addition on the unit circle, we are done. See \cite{mwe}.
\end{proof}

In practice, the examples of families $\{L_i\},\{R_j\}$ as above come from the mutually unbiased bases (MUB), via the general theory developed in \cite{deb}. We are particularly interested here in a well-known Fourier type construction, as follows:

\begin{proposition}
For $q\geq3$ prime, the matrices $\{F_q,DF_q,\ldots,D^{q-1}F_q\}$, where 
$$D=diag\left(1,1,w,w^3,w^6,w^{10},\ldots,w^{\frac{q^2-1}{8}},\ldots,w^{10},w^6,w^3,w\right)$$
with $w=e^{2\pi i/q}$, are such that $\frac{1}{\sqrt{q}}E_i^*E_j$ is complex Hadamard, for any $i\neq j$.
\end{proposition}

\begin{proof}
With $0,1,\ldots,q-1$ as indices, the formula of the above matrix $D$ is:
$$D_c=w^{0+1+\ldots+(c-1)}=w^{\frac{c(c-1)}{2}}$$

Since we have $\frac{1}{\sqrt{q}}E_i^*E_j\in\sqrt{q}U_q$, we just need to check that these matrices have entries belonging to $\mathbb T$, for any $i\neq j$. With $k=j-i$, these entries are given by:
$$\frac{1}{\sqrt{q}}(E_i^*E_j)_{ab}=\frac{1}{\sqrt{q}}(F_q^*D^kF_q)_{ab}=\frac{1}{\sqrt{q}}\sum_cw^{c(b-a)}D_c^k$$

Now observe that with $s=b-a$, we have the following formula:
\begin{eqnarray*}
\left|\sum_cw^{cs}D_c^k\right|^2
&=&\sum_{cd}w^{cs-ds}w^{\frac{c(c-1)}{2}\cdot k-\frac{d(d-1)}{2}\cdot k}
=\sum_{cd}w^{(c-d)\left(\frac{c+d-1}{2}\cdot k+s\right)}\\
&=&\sum_{de}w^{e\left(\frac{2d+e-1}{2}\cdot k+s\right)}
=\sum_e\left(w^{\frac{e(e-1)}{2}\cdot k+es}\sum_dw^{edk}\right)\\
&=&\sum_ew^{\frac{e(e-1)}{2}\cdot k+es}\cdot q\delta_{e0}=q
\end{eqnarray*}

Thus the entries are on the unit circle, and we are done.
\end{proof}

We recall that the Legendre symbol is defined as follows:
$$\left(\frac{s}{q}\right)=\begin{cases}
0&{\rm if}\ s=0\\
1&{\rm if}\ \exists\,\alpha,s=\alpha^2\\
-1&{\rm if}\not\exists\,\alpha,s=\alpha^2
\end{cases}$$

Here, and in what follows, all the numbers are taken modulo $q$. We have:

\begin{proposition}
The matrices $G_k=\frac{1}{\sqrt{q}}F_q^*D^kF_q$, with $D=diag(w^{\frac{c(c-1)}{2}})$ being as above, and with $k\neq0$ are circulant, their first row vectors $V^k$ being given by
$$V^k_i=\delta_q\left(\frac{k/2}{q}\right)w^{\frac{q^2-1}{8}\cdot k}\cdot w^{-\frac{\frac{i}{k}(\frac{i}{k}-1)}{2}}$$
where $\delta_q=1$ if $q=1(4)$ and $\delta_q=i$ if $q=3(4)$, and with all inverses being taken in $\mathbb Z_q$.
\end{proposition}

\begin{proof}
This is a standard exercice on quadratic Gauss sums. First of all, the matrices $G_k$ in the statement are indeed circulant, their first vectors being given by:
$$V^k_i=\frac{1}{\sqrt{q}}\sum_cw^{\frac{c(c-1)}{2}\cdot k+ic}$$

Let us first compute the square of this quantity. We have:
$$(V_i^k)^2=\frac{1}{q}\sum_{cd}w^{\left[\frac{c(c-1)}{2}+\frac{d(d-1)}{2}\right]k+i(c+d)}$$

The point now is that the sum $S$ on the right, which has $q^2$ terms, decomposes as follows, where $x$ is a certain exponent, depending on $q,i,k$:
$$S=\begin{cases}
(q-1)(1+w+\ldots+w^{q-1})+qw^x&{\rm if}\ q=1(4)\\
(q+1)(1+w+\ldots+w^{q-1})-qw^x&{\rm if}\ q=3(4)
\end{cases}$$

We conclude that we have a formula as follows, where $\delta_q\in\{1,i\}$ is as in the statement, so that $\delta_q^2\in\{1,-1\}$ is given by $\delta_q^2=1$ if $q=1(4)$ and $\delta_q^2=-1$ if $q=3(4)$:
$$(V_i^k)^2=\delta_q^2\,w^x$$

In order to compute now the exponent $x$, we must go back to the above calculation of the sum $S$.  We succesively have:

-- First of all, at $k=1,i=0$ we have $x=\frac{q^2-1}{4}$.

-- By translation we obtain $x=\frac{q^2-1}{4}-i(i-1)$, at $k=1$ and any $i$.

-- By replacing $w\to w^k$ we obtain $x=\frac{q^2-1}{4}\cdot k-\frac{i}{k}(\frac{i}{k}-1)$, at any $k\neq0$ and any $i$. 

Summarizing, we have computed the square of the quantity that we are interested in, the formula being as follows, with $\delta_q$ being as in the statement:
$$(V^k_i)^2=\delta_q^2\cdot w^{\frac{q^2-1}{4}\cdot k}\cdot w^{-\frac{i}{k}(\frac{i}{k}-1)}$$

By extracting now the square root, we obtain a formula as follows:
$$V^k_i=\pm\delta_q\cdot w^{\frac{q^2-1}{8}\cdot k}\cdot w^{-\frac{\frac{i}{k}(\frac{i}{k}-1)}{2}}$$

The computation of the missing sign is non-trivial, but by using the theory of quadratic Gauss sums, and more specifically a result of Gauss, computing precisely this kind of sign, we conclude that we have indeed a Legendre symbol, $\pm=\left(\frac{k/2}{q}\right)$, as claimed. 
\end{proof}

Let us combine now the above results. We obtain the following statement:

\begin{theorem}
Let $q\geq3$ be prime, consider two subsets $S,T\subset\{0,1,\ldots,q-1\}$ satisfying $S\cap T=\emptyset$, and write $S=\{s_1,\ldots,s_M\}$ and $T=\{t_1,\ldots,t_N\}$. The matrix
$$H_{ia,jb}=K_{ij}V^{t_j-s_i}_{b-a}$$
where $V$ is as above, is then partial Hadamard, provided that $K\in M_{M\times N}(\mathbb C)$ is.
\end{theorem}

\begin{proof}
This follows indeed by using the general construction in Proposition 4.1 above, with input coming from Proposition 4.2 and Proposition 4.3.
\end{proof}

In order to find now isolated matrices, the idea from \cite{mwe} is that of starting with an isolated matrix $K\in M_{M\times N}(\mathbb C)$, such as the Fourier one $F_p\in M_p(\mathbb C)$, or one of its truncations coming from Conjecture 1.14, and then use suitable sets $S,T$ as above.

The defect computations are, however, quite difficult. In the context of Proposition 4.3 above, the Gauss symbols $\delta_q$ will cancel, according to $\delta_q\bar{\delta}_q=1$, but we are left with a quite complicated computation, involving the Legendre symbols, and roots of unity. 

As a concrete statement, however, we have the following conjecture:

\begin{conjecture}[Arithmetic isolation]
The matrix constructed in Theorem 4.4 is isolated, provided that $K$ is an isolated truncation of $F_p$, with $p$ prime, and $S,T$ consist respectively of consecutive odd numbers, and consecutive even numbers.
\end{conjecture}

This statement is supported by several computer-aided computations in \cite{mwe}, in the square matrix case, and by some specific rectangular matrix computations as well.

\section{Quantum permutations}

In the reminder of this paper we discuss some related quantum algebraic aspects. As explained in \cite{bsk}, the partial Hadamard matrices have something to do with the partial permutations. In order to explain this material, our starting point will be:

\begin{definition}
A partial permutation of $\{1\,\ldots,N\}$ is a bijection $\sigma:X\simeq Y$, with $X,Y\subset\{1,\ldots,N\}$. We denote by $\widetilde{S}_N$ the set formed by such partial permutations. 
\end{definition}

Observe that we have $S_N\subset\widetilde{S}_N$. The embedding $u:S_N\subset M_N(0,1)$ given by permutation matrices can be extended to an embedding $u:\widetilde{S}_N\subset M_N(0,1)$, as follows:
$$u_{ij}(\sigma)=
\begin{cases}
1&{\rm if}\ \sigma(j)=i\\
0&{\rm otherwise}
\end{cases}$$

By looking at the image of this embedding, we see that $\widetilde{S}_N$ is in bijection with the matrices $M\in M_N(0,1)$ having at most one 1 entry on each row and column.

Regarding the ``quantum version'' of the above notions, following the paper of Wang \cite{wa1}, where a free analogue of the permutation group $S_N$ was constructed, we have:

\begin{definition}
A submagic matrix is a matrix $u\in M_N(A)$ whose entries are projections ($p^2=p^*=p$), which are pairwise orthogonal on rows and columns. We let $C(\widetilde{S}_N^+)$ be the universal $C^*$-algebra generated by the entries of a $N\times N$ submagic matrix. 
\end{definition}

The algebra $C(\widetilde{S}_N^+)$ has a comultiplication given by $\Delta(u_{ij})=\sum_ku_{ik}\otimes u_{kj}$, and a counit given by $\varepsilon(u_{ij})=\delta_{ij}$. Thus $\widetilde{S}_N^+$ is a compact quantum semigroup. See \cite{bsk}.

The relation with the partial Hadamard matrices is immediate, coming from:

\begin{proposition}
If $H\in M_{M\times N}(\mathbb T)$ is partial Hadamard, with rows $H_1,\ldots,H_M\in\mathbb T^N$, then the matrix of rank one projections $P_{ij}=Proj(H_i/H_j)$ is submagic. Thus $H$ produces a $C^*$-algebra  representation $\pi_H:C(\widetilde{S}_M^+)\to M_N(\mathbb C)$, given by $u_{ij}\to P_{ij}$.
\end{proposition}

\begin{proof}
We have indeed the following computation:
$$\Big\langle\frac{H_i}{H_j},\frac{H_i}{H_k}\Big\rangle=\sum_l\frac{H_{il}}{H_{jl}}\cdot\frac{H_{kl}}{H_{il}}=\sum_l\frac{H_{kl}}{H_{jl}}=<H_k,H_j>=\delta_{jk}$$

A similar computation works for the columns of $P$, and this gives the result.
\end{proof}

We can further build on this simple observation, in the following way:

\begin{definition}
The minimal semigroup $G\subset\widetilde{S}_M^+$ producing a factorization of type
$$\xymatrix{C(\widetilde{S}_{M}^+)\ar[rr]^{\pi_H}\ar[rd]&&M_N(\mathbb C)\\&C(G)\ar[ur]&}$$
with a bialgebra map at left, is called quantum semigroup associated to $H$.
\end{definition}

Here the fact that the bialgebra $C(G)$ as above exists, and is unique, follows by dividing the algebra $C(\widetilde{S}_{M}^+)$ by a suitable ideal. For full details, we refer once again to \cite{bsk}.

In order to discuss the ``classical'' case, we will need:

\begin{definition}
A pre-Latin square is a matrix $L\in M_M(1,\ldots,N)$ having the property that its entries are distinct, on each row and each column. 
\end{definition}

Given such a matrix $L$, to any $x\in\{1,\ldots,N\}$ we can associate the partial permutation $\sigma_x\in\widetilde{S}_M$ given by $\sigma_x(j)=i\iff L_{ij}=x$. We denote by $G\subset\widetilde{S}_M$ the semigroup generated by $\sigma_1,\ldots,\sigma_N$, and call it semigroup associated to $L$.

Also, given an orthogonal basis $\xi=(\xi_1,\ldots,\xi_N)$ of $\mathbb C^N$, we can construct a submagic matrix $P\in M_M(M_N(\mathbb C))$, according to the formula $P_{ij}=Proj(\xi_{L_{ij}})$.

With these notations, we have the following result, from \cite{bsk}:

\begin{proposition}
If $H\in M_{M\times N}(\mathbb C)$ is partial Hadamard, the following are equivalent:
\begin{enumerate}
\item The semigroup $G\subset\widetilde{S}_M^+$ is classical, i.e. $G\subset\widetilde{S}_M$.

\item The projections $P_{ij}=Proj(H_i/H_j)$ pairwise commute.

\item The vectors $H_i/H_j\in\mathbb T^N$ are pairwise proportional, or orthogonal. 

\item The submagic matrix $P=(P_{ij})$ comes for a pre-Latin square $L$.
\end{enumerate}
In addition, if so is the case, $G$ is the semigroup associated to $L$.
\end{proposition}

\begin{proof}
Here $(1)\iff(2)$ is clear, $(2)\iff(3)$ comes from the fact that two rank 1 projections commute precisely when their images coincide, or are orthogonal, $(3)\iff(4)$ is clear again, and the last assertion comes from Gelfand duality. See \cite{bsk}. 
\end{proof}

We call ``classical'' the matrices in Proposition 5.6, that we will study now.

Let us begin with a study at $M=2$. With $\widetilde{S}_2=\{id,\tau,11,12,21,22,\emptyset\}$, where $\tau$ is the transposition, $ij$ is the partial permutation $i\to j$, and $\emptyset$ is the null map, we have:

\begin{proposition}
A partial Hadamard matrix $H\in M_{2\times N}(\mathbb T)$, in dephased form
$$H=\begin{pmatrix}1&\ldots&1\\ \lambda_1&\ldots&\lambda_N\end{pmatrix}$$
is of classical type when one of the following happens:
\begin{enumerate}
\item Either $\lambda_i=\pm w$, for some $w\in\mathbb T$, in which case $G=\{id,\tau\}$.

\item Or $\sum_i\lambda_i^2=0$, in which case $G=\{id,11,12,21,22,\emptyset\}$
\end{enumerate}
\end{proposition}

\begin{proof}
With $1=(1,\ldots,1)$ and $\lambda=(\lambda_1,\ldots,\lambda_N)$, the matrix formed by the vectors $H_i/H_j$ is $(^1_{\bar{\lambda}}{\ }^\lambda_1)$. Since $1\perp\lambda,\bar{\lambda}$ we just have to compare $\lambda,\bar{\lambda}$, and we have two cases:

(1) Case $\lambda\sim\bar{\lambda}$. This means $\lambda^2\sim1$, and so $\lambda_i=\pm w$, for some $w\in\mathbb T$. In this case the associated pre-Latin square is $L=(^1_2{\ }^2_1)$, the partial permutations $\sigma_x$ associated to $L$ are $\sigma_1=id$ and $\sigma_2=\tau$, and we have $G=<id,\tau>=\{id,\tau\}$, as claimed.

(2) Case $\lambda\perp\bar{\lambda}$. This means $\sum_i\lambda_i^2=0$. In this case the associated pre-Latin square is $L=(^1_3{\ }^2_1)$, the associated partial permutations $\sigma_x$ are given by $\sigma_1=id$, $\sigma_2=21$, $\sigma_3=12$, and so we obtain $G=<id,21,12>=\{id,11,12,21,22,\emptyset\}$, as claimed.
\end{proof}

The matrices in (1) are, modulo equivalence, those which are real. As for the matrices in (2), these are parametrized by the solutions $\lambda\in\mathbb T^N$ of the following equations:
$$\sum_i\lambda_i=\sum_i\lambda_i^2=0$$

These equations are obviously something quite complicated. So, as a conclusion, the world of classical Hadamard matrices is quite subtle, in the rectangular setting.

Let us discuss now the truncated Fourier matrix case. First, we have:

\begin{proposition}
The Fourier matrix, $F_N=(w^{ij})$ with $w=e^{2\pi i/N}$, is of classical type, and the associated group $G\subset S_N$ is the cyclic group $\mathbb Z_N$.
\end{proposition}

\begin{proof}
Since $H=F_N$ is a square matrix, the associated semigroup $G\subset\widetilde{S}_N^+$ must be a quantum group, $G\subset S_N^+$. We must prove that this quantum group is $G=\mathbb Z_N$.

With $\rho=(1,w,w^2,\ldots,w^{N-1})$ the rows of $H$ are given by $H_i=\rho^i$, and so we have $H_i/H_j=\rho^{i-j}$. We conclude that $H$ is indeed of classical type, coming from the Latin square $L_{ij}=j-i$ and from the orthogonal basis $\xi=(1,\rho^{-1},\rho^{-2},\ldots,\rho^{1-N})$.

We have $G=<\sigma_1,\ldots,\sigma_N>$, where $\sigma_x\in S_N$ is given by $\sigma_x(j)=i\iff L_{ij}=x$. From $L_{ij}=j-i$ we obtain $\sigma_x(j)=j-x$, and so $G=\{\sigma_1,\ldots,\sigma_N\}\simeq\mathbb Z_N$, as claimed.
\end{proof}

For a partial permutation $\sigma:I\simeq J$ with $|I|=|J|=k$, we set $\kappa(\sigma)=k$.  Also, we denote by $F_{M,N}$ be the upper $M\times N$ submatrix of the Fourier matrix $F_N$.

With these conventions, we have the following result:

\begin{theorem}
The semigroup $G_{M,N}\subset\widetilde{S}_M$ associated to $F_{M,N}$ is as follows:
\begin{enumerate}
\item In the $N>2M-2$ regime, $G_{M,N}\subset\widetilde{S}_M$ is formed by the maps $\sigma:I\simeq J$, $\sigma(j)=j-x$, with $I,J\subset\{1,\ldots,M\}$ being intervals.

\item The components $G_{M,N}^{(k)}=\{\sigma\in G_{M,N}|\kappa(\sigma)=k\}$ with $k>2M-N$ are, in the $M<N\leq2M-2$ regime, the same as those in the $N>2M-2$ regime.
\end{enumerate}
\end{theorem}

\begin{proof}
Both the assertions follow from an elementary study, as follows:

(1) Since for $\widetilde{H}=F_N$ the associated Latin square is circulant, $\widetilde{L}_{ij}=j-i$, the pre-Latin square that we are interested in is:
$$L=\begin{pmatrix}
0&1&2&\ldots&M-1\\
N-1&0&1&\ldots&M-2\\
N-2&N-1&0&\ldots&M-3\\
\ldots\\
N-M+1&N-M+2&N-M+3&\ldots&0
\end{pmatrix}$$

Observe that, due to our $N>2M-2$ assumption, we have $N-M+1>M-1$, and so the entries above the diagonal are distinct from those below the diagonal.

Let us compute now the partial permutations $\sigma_x\in\widetilde{S}_M$ given by $\sigma_x(j)=i\iff L_{ij}=x$. We have $\sigma_0=id$, and then $\sigma_2=\sigma_1^2,\sigma_3=\sigma_1^3,\ldots,\sigma_{M-1}=\sigma_1^{M-1}$. As for the remaining partial permutations, these are given by $\sigma_{N-1}=\sigma_1^{-1},\sigma_{N-2}=\sigma_2^{-1},\ldots,\sigma_{N-M+1}=\sigma_{M-1}^{-1}$. Thus $G_{M,N}=<\sigma_1>$. Now if we denote by $G_{M,N}'$ the semigroup in the statement, we have $\sigma_1\in G_{M,N}'$, so $G_{M,N}\subset G_{M,N}'$. The reverse inclusion can be proved as follows:

(a) Assume first that $\sigma\in G_{M,N}'$, $\sigma:I\simeq J$ has the property $M\in I,J$. Then we can write $\sigma=\sigma_{N-k}\sigma_k$, with $k=M-|I|$, so we have $\sigma\in G_{M,N}$.

(b) Assume now that $\sigma\in G_{M,N}'$, $\sigma:I\simeq J$ has just the property $M\in I$ or $M\in J$. In this case we have as well $\sigma\in G_{M,N}$, because $\sigma$ appears from one of the maps in (a) by adding a ``slope'', which can be obtained by composing with a suitable map $\sigma_k$.

(c) Assume now that $\sigma\in G_{M,N}'$, $\sigma:I\simeq J$ is arbitrary. Then we can write $\sigma=\sigma'\sigma''$ with $\sigma':L\simeq J$, $\sigma'':I\simeq L$, where $L$ is an interval satisfying $|L|=|I|=|J|$ and $M\in L$, and since $\sigma',\sigma''\in G_{M,N}$ by (b), we are done.

(2) At $M<N\leq2M-2$ the pre-Latin square has 0 on the diagonal, and then takes its entries from the following set, in a uniform way from each of the 3 components:
$$S=\{1,\ldots,N-M\}\cup\{N-M+1,\ldots,M-1\}\cup\{M,\ldots,N-1\}$$ 

The point now is that $\sigma_1,\ldots,\sigma_{N-M}$ are given by the same formulae as those in the proof of (1) above, then $\sigma_{N-M+1},\ldots,\sigma_{M-1}$ all satisfy $\kappa(\sigma)=2M-N$, and finally $\sigma_M,\ldots,\sigma_{N-1}$ are once again given by the formulae in the proof of (1) above.

Now since we have $\kappa(\sigma\rho)\leq\min(\kappa(\sigma),\kappa(\rho))$, adding the maps $\sigma_{N-M+1},\ldots,\sigma_{M-1}$ to the semigroup $G_{M,N}\subset\widetilde{S}_M$ computed in the proof of (1) above won't change the $G_{M,N}^{(k)}$ components of this semigroup at $k>2M-N$, and this gives the result.
\end{proof}

Further improving this result, as to have a full description of the semigroup $G_{M,N}\subset\widetilde{S}_M$ associated to $F_{M,N}$ is an interesting question, that we would like to raise here.

\section{Isolation questions}

In this section we briefly discuss isolation questions for the quantum algebraic objects constructed in section 5. In the square matrix case some probabilistic technology is available from \cite{wa2}, and for the Hadamard representations, the result is as follows:

\begin{theorem}
The $p$-th moment of spectral measure of the quantum group $G\subset S_N^+$ associated to an Hadamard matrix $H\in M_N(\mathbb C)$ is the dimension of the $1$-eigenspace of
$$(T_p)_{i_1\ldots i_p,j_1\ldots j_p}=tr(P_{i_1j_1}\ldots P_{i_pj_p})$$
where $P_{ij}=\frac{1}{N}(H_{ik}\bar{H}_{jk}H_{jl}\bar{H}_{il})_{kl}$ denotes as usual the projection on $H_i/H_j$.
\end{theorem}

\begin{proof}
This follows indeed from the general moment formula established in  \cite{wa2}, with input coming from Proposition 5.3 above. For full details here, we refer to \cite{bbi}.
\end{proof}

The above result suggests the following definition:

\begin{definition}
A one-parameter family of compact quantum groups $G_q\subset S_N^+$ is called a ``deformation'' when the moment formula in Theorem 6.1
holds, for a certain family of matrices $\{T_p^q|p\in\mathbb N\}$, depending continuously on $q\in\mathbb T$. 
\end{definition}

Note that the representation theory of $G_q$ can vary wildly with $q\in\mathbb T$, as shown for instance in \cite{bbi}, \cite{bic}. We do not know whether the above notion of deformation admits or not a simpler axiomatization. This is a question that we would like to raise here.

Yet another series of interesting questions comes in relation with the almost Hadamard matrices. These matrices were introduced in \cite{bnz} in the real setting, the motivation being the fact that such matrices exist at any $N\in\mathbb N$, and can therefore replace the usual Hadamard matrices, when these are not available. A complex version of the theory was recently discussed in \cite{bn2}, with the following quite surprizing conclusion:

\begin{conjecture}[Almost Hadamard conjecture]
Any almost Hadamard matrix, taken in a complex sense, must be actually a complex Hadamard matrix.
\end{conjecture}

In more concrete terms, the problem is that of proving that any local maximizer of the 1-norm on the unitary group $U_N$ is actually a global maximizer. See \cite{bn2}.

The above conjecture is something definitely geometric, and among the pieces of evidence for it are some defect computations for the almost Hadamard matrices, leading to the same equations as those found by Tadej and \.Zyczkowski in \cite{tz2}. These computations suggest the following conjecture, which was made as well in \cite{bn2}:

\begin{conjecture}[AHM isolation]
For a rescaled unitary matrix $U\in\sqrt{N}U_N$, the following conditions are equivalent:
\begin{enumerate}
\item $U$ is a strict almost Hadamard matrix.

\item $U$ is an isolated complex Hadamard matrix.

\item $U$ is a complex Hadamard matrix with minimal AHM/CHM defect.
\end{enumerate}
\end{conjecture}

We refer to \cite{bn2} for further details on this conjecture. We believe that there is a relation here with the above quantum group considerations, coming from a theory of defect for the ``flat'' magic unitary matrices, in the sense of \cite{bn1}. However, this defect theory is not developed yet, and we have no further results on this subject, for the moment.

\end{document}